\documentclass[10pt]{article}
\usepackage{basepreamble}

\usepackage{amsfonts}
\usepackage{amsmath}
\usepackage{enumerate}
\usepackage{url}

\newcommand{\norm}[1]{\left|{#1}\right|}

\newcommand{\ceil}[1]{\lceil {#1} \rceil}

\newcommand{\seq}[1]{\{#1\}}

\input xy
\xyoption{all}

\begin{document}

\title{A note on the finitization of Abelian and Tauberian theorems}

\author{Thomas Powell}

\maketitle 

\begin{abstract}
We present finitary formulations of two well known results concerning infinite series, namely Abel's theorem, which establishes that if a series converges to some limit then its Abel sum converges to the same limit, and Tauber's theorem, which presents a simple condition under which the converse holds. Our approach is inspired by proof theory, and in particular G\"{o}del's functional interpretation, which we use to establish quantitative version of both of these results.
\end{abstract}

\section{Introduction}

In an essay of 2007 \cite{Tao(2007.0)} (later published as part of \cite{Tao(2008.0)}) T. Tao discussed the so-called \emph{correspondence principle} between `soft' and `hard' analysis, whereby many \emph{infinitary} notions from analysis can be given an equivalent \emph{finitary} formulation. An important instance of this phenomenon is provided by the simple concept of Cauchy convergence of a sequence $\seq{c_n}$:
\begin{equation*}
\forall \varepsilon>0\; \exists N\; \forall m,n\geq N\; (|c_m-c_n|\leq \varepsilon).
\end{equation*}
This corresponds to the finitary notion of $\seq{c_n}$ being \emph{metastable}, which is given by the following formula:
\begin{equation}
\label{eqn-metastable}
\forall \varepsilon>0\; \forall g:\NN\to\NN\; \exists N\; \forall m,n\in [N;N+g(N)]\; (|c_m-c_n|\leq \varepsilon),
\end{equation}
where $[N;N+k]:=\{N,N+1,\ldots,N+k-1,N+k\}$. Roughly speaking, a sequence $\seq{c_n}$ is metastable if for any given error $\varepsilon>0$ it contains a finite regions of stability of any `size', where size is represented by the function $g:\NN\to\NN$.
 
The equivalence of Cauchy convergence and metastability is established via purely logical reasoning, and indeed, as was quickly observed, the correspondence principle as presented in \cite{Tao(2007.0)} has deep connections with proof theory. More specifically, the finitary variant of an infinitary statement is typically closely related to its \emph{classical Dialectica interpretation} \cite{AvFef(1998.0)}, which provides a general method for obtaining quantitative versions of mathematical theorems.

Finitary formulations of infinitary properties play a central role in the \emph{proof mining} program developed by U. Kohlenbach from the early 90s \cite{Kohlenbach(2008.0)}. Here, it is often the case that a given mathematical theorem has, in general, no computable realizer (for Cauchy convergence this is demonstrated by the existence of so-called \emph{Specker sequences} \cite{Specker(1949.0)}, which will be discussed further in Section \ref{sec-specker}). On the other hand, the corresponding finitary formulation can typically not only be realized, but a realizer can be directly extracted from a proof that the original property holds. The extraction of a computable bound $\Omega(\varepsilon,g)$ on $N$ in (\ref{eqn-metastable}) -- a so-called \emph{rate of metastability} -- is a standard result in this area (e.g. \cite{Kohlenbach(2005.1),KohKou(2015.0),KohLeu(2012.0),Powell(2019.2)}), and techniques from proof theory are often used to give finitizations of more complex statements, including for example the \emph{Bolzano-Weierstrass theorem} \cite{KohSaf(2010.0)} and \emph{Ramsey's theorem} \cite{KohKr(2009.0),OliPow(2015.1)}.

In this article, we apply the aforementioned ideas to study the relationship between two distinct forms of convergence from a finitary perspective, namely (I) the convergence of an infinite series of reals $\seq{a_n}$, and (II) the limit as $x\to 1^{-}$ of the power series it generates i.e.
\begin{itemize}

\item[(I)] $s_n:=\sum_{i=0}^n a_i$ as $n\to\infty$, and

\item[(II)] $F(x):=\sum_{i=0}^\infty a_ix^i$ as $x\to 1^{-}$.

\end{itemize}
Several classic results apply here. \emph{Abel's theorem} covers one direction, and states that if $\lim_{n\to\infty}s_n=s$ then we also have $\lim_{x\to 1^{-}} F(x)=s$. The converse also holds, subject to the additional condition that $a_n=o(1/n)$, a result due to A. Tauber which is has since become the first and simplest instance of a whole class of results known as \emph{Tauberian theorems}. In this article, we provide new quantitative versions of these two theorems, which take the shape of a route between various forms of metastability. Our formulations of these results are entirely finitary in nature, and capture the combinatorial core of the original proofs. The infinitary theorems can then be directly derived from ours in a uniform way, using purely logical reasoning.

Though both Abel's and Tauber's theorems are elementary to state and prove, establishing in each case a natural finitary formulation from which the original theorem can be rederived is non-trivial, as is generally the case when it comes to correctly finitizing infinitary statements (an illuminating discussion of the subtleties which arise from the similarly elementary \emph{infinite pigeonhole principle} is given in \cite{GasKoh(2010.0)}). We begin by establishing Cauchy variants of Abel and Tauber's theorems which do not explicitly mention limit points. We then demonstrate that Specker phenomena propagate through both theorems, and as such, in order to give quantitative versions we are forced to consider metastable variants of the associated limiting processes. We then state and prove our finitary theorems, and demonstrate how the original theorems can be reobtained in a uniform way.

There are two main motivating factors behind this short article. The first is the fact that Abelian and Tauberian theorems give rise to simple and yet illuminating examples of the correspondence principle and related concepts such as metastability, which can be presented in such a way that we are not required to explicitly introduce any proof theoretic concepts (indeed, even the notion of a higher order functional is only needed in Section \ref{sec-reobtain} to rederive the original results). As such, it is hoped that our analysis will be of interest to a general mathematical audience. A brief note on the underlying proof theory and connections with e.g. G\"{o}del's Dialectica interpretation is provided in Section \ref{sec-prooftheory}, but this is not required in order to follow the main part of the paper. 

More importantly though, we consider the relatively simple results here as paving the way for a more advanced study of theorems of Abelian or Tauberian type, of which those studied here are the simplest. In particular, Tauber's theorem was significantly generalised by Hardy and Littlewood \cite{Titchmarsh(1939.0)} and then by Wiener \cite{Wiener(1932.0)} (see e.g. \cite{Korevaar(2004.0)} for a modern survey covering these and later developments). We conjecture that a wealth of interesting case studies for applied proof theory can be found in this area, and hope that in this article to have taken a first step in this direction.

\section{Cauchy variants of convergence properties}
\label{sec-cauchy}

We start off with some preliminary mathematical results, with the aim of setting up suitable Cauchy formulations of both Abel's and Tauber's theorem, which will then be analysed over the remainder of the paper. For the sake of completeness, we begin by stating these theorems as they are usually formulated.
\begin{theorem}
[Abel's theorem]
\label{thm-abel}
Let $\seq{a_n}$ be a sequence of reals and suppose that $\sum_{i=0}^\infty a_i=s$. Then $\lim_{x\to 1^{-}}\sum_{i=0}^\infty a_ix^i=s$.
\end{theorem}
\begin{theorem}
[Tauber's theorem]
\label{thm-taub}
Let $\seq{a_n}$ be a sequence of reals with $a_n=o(1/n)$ and suppose that $\lim_{x\to 1^{-}}\sum_{i=0}^\infty a_ix^i=s$. Then $\sum_{i=0}^\infty a_i=s$.
\end{theorem}
Our preference for Cauchy variants of these results lies in the fact that we do not have to directly deal with limits, making the quantifier complexity of the underlying notions of convergence significantly simpler. In particular, we want to formulate the statement that $\lim_{n\to \infty} s_n=\lim_{x\to ^{-}} F(x)$ without mentioning either of the limits explicitly.
%
%
%
\begin{lemma}
\label{lem-eqcauchy}
Let $F:[0,1)\to\RR$ be a function and $\seq{s_n}$ a sequence of reals. Then each of the following implies $\lim_{n\to\infty} s_n=\lim_{x\to 1^{-}} F(x)$:
\begin{enumerate}[(i)]

\item $\seq{s_n}$ converges and $\lim_{m,n\to\infty}|F(x_m)-s_n|=0$ for all $\seq{x_m}$ in $[0,1)$ with $\lim_{m\to\infty} x_m=1$.

\item $\lim_{x\to 1^{-}} F(x)$ exists and $\lim_{m,n\to\infty}|F(y_m)-s_n|=0$ for some $\seq{y_m}$ in $[0,1)$ with $\lim_{m\to\infty} y_n=1$.

\end{enumerate}
\end{lemma}

\begin{proof}
From (i), suppose that $\lim_{n\to\infty} s_n=s$ and $\lim_{m\to\infty} x_m=1$. Then for any $\varepsilon>0$ there is a sufficiently large $N$ such that $|F(x_m)-s_n|\leq\tfrac{\varepsilon}{2}$ for all $m,n\geq N$ and $N'$ such that $|s_n-s|\leq\tfrac{\varepsilon}{2}$ for all $n\geq N'$, and thus $|F(x_m)-s|\leq\varepsilon$ for $m\geq N$. Since $\{x_m\}$ was arbitrary, we have shown that $\lim_{m\to\infty} F(x_m)= s$ whenever $\lim_{m\to\infty}x_m= 1$ and thus $\lim_{x\to 1^{-}}F(x)=s$. Similarly, from (ii), if $\lim_{x\to 1^{-}}F(x)=s$ then in particular $F(y_m)\to s$ and thus for any $\varepsilon>0$ there is some $N$ with $|F(y_m)-s|\leq\tfrac{\varepsilon}{2}$ for all $m\geq N$ and some $N'$ with $|F(y_m)-s_n|\leq\tfrac{\varepsilon}{2}$ for all $m,n\geq N'$. Therefore $|s_n-s|\leq\varepsilon$ for all $n\geq N'$, and so we have shown that $\lim_{n\to\infty} s_n=s$.
\end{proof}


%
We are now able to give Cauchy variants to both main theorems. For the remainder of the paper, we define $\seq{s_n}$ and $F:[0,1)\to\RR$ as
\begin{equation*}
s_n:=\sum_{i=0}^n a_i \ \ \ \mathrm{and} \ \ \ F(x):=\sum_{i=0}^\infty a_ix^i
\end{equation*}
where $\seq{a_n}$ will be some given sequence of real numbers.

\begin{theorem}
[Abel's theorem, Cauchy variant]
\label{thm-abelcauchy}
Let $\seq{a_n}$ be a sequence of reals such that $\{s_n\}$ is Cauchy. Then $\lim_{m,n\to\infty}|F(x_m)-s_n|=0$ whenever $\lim_{m\to\infty} x_m=1$.
\end{theorem}
To see that Theorem \ref{thm-abelcauchy} implies Theorem \ref{thm-abel}, suppose that $\lim_{n\to\infty} s_n=s$. Then in particular $\seq{s_n}$ is Cauchy, and so $\lim_{m,n\to\infty}|F(x_m)-s_n|=0$ whenever $\lim_{m\to\infty} x_m=1$. But then by Lemma \ref{lem-eqcauchy} we have $\lim_{x\to 1^{-}} F(x)=s$.
\begin{theorem}
[Tauber's theorem, Cauchy variant]
\label{thm-taubcauchy}
Let $\seq{a_n}$ be a sequence of reals with $a_n=o(1/n)$ and suppose that $\seq{F(v_n)}$ is Cauchy, where $v_n:=1-\tfrac{1}{n}$. Then $\lim_{m,n\to\infty}|F(v_m)-s_n|=0$.
\end{theorem}
That Theorem \ref{thm-taubcauchy} implies Theorem \ref{thm-taub} is similar: Suppose that $a_n=o(1/n)$ and $\lim_{x\to 1^{-}} F(x)=s$. Since $\lim_{m\to\infty} v_m=1$ then $\lim_{m\to\infty} F(v_m)=s$ and so in particular $\seq{F(v_m)}$ is Cauchy, and thus $\lim_{m,n\to\infty}|F(v_m)-s_n|=0$. But then by Lemma \ref{lem-eqcauchy} (ii) it follows that $\lim_{n\to\infty}s_n=s$.

\section{On Specker sequences}
\label{sec-specker}

In this short section, we show that Specker phenomena propagate through both Theorems \ref{thm-abelcauchy} and \ref{thm-taubcauchy}. For the former this is completely straightforward, but for the latter a little care is needed to construct a suitable sequence satisfying the Tauber condition $a_n=o(1/n)$. These results confirms that we cannot hope to produce general quantitative formulations of either theorem that provide a direct rate of convergence for the conclusion, and so we must instead make use of the relevant notions of \emph{metastability}. 

We first recall that a Specker sequence, first introduced in \cite{Specker(1949.0)}, is a \emph{computable}, \emph{monotonically increasing} and \emph{bounded} sequence of rationals $\seq{q_n}$ whose limit is not a computable real number. What this means in practice is that the sequence possess neither a computable \emph{rate of convergence} nor a computable \emph{rate of Cauchy convergence}, where by the latter we mean a computable function $\pi:\QQ_+\to \NN$ satisfying
\begin{equation*}
\forall \varepsilon\in\QQ_+, \forall m,n\geq \pi(\varepsilon)(|q_m-q_n|\leq \varepsilon),
\end{equation*}
where here $\QQ_+$ denotes the set of all strictly positive rationals.
\begin{proposition}
\label{prop-abelspecker}
There exists some $\seq{a_n}$ satisfying the premise of Theorem \ref{thm-abelcauchy}, whereby for any $\seq{x_n}$ in $[0,1)$ with $\lim_{n\to\infty} x_m=1$, though $\lim_{m,n\to\infty}|F(x_m)-s_n|=0$ this has no computable rate of convergence i.e. a computable function $\pi:\QQ_+\to\NN$ satisfying
\begin{equation*}
\forall \varepsilon\in\QQ_+,\forall m,n\geq \pi(\varepsilon)(|F(x_m)-s_n|\leq \varepsilon).
\end{equation*}
\end{proposition}

\begin{proof}
Take any Specker sequence $\seq{q_n}$ and define $a_0:=q_0$ and $a_{n+1}:=q_{n+1}-q_n$, so that $s_n=q_n$, and so by definition $\seq{s_n}$ is Cauchy. Suppose for contradiction there exists some $\seq{x_m}$ with $\lim_{m\to\infty} x_m=1$ such that $|F(x_m)-s_n|\to 0$ as $m,n\to\infty$ with a computable rate of convergence $\pi$. Then for any $\varepsilon\in\QQ_+$ we have
\begin{equation*}
|s_m-s_n|\leq |s_m-F(x_{\pi(\varepsilon/2)})|+|F(x_{\pi(\varepsilon/2)})-s_n|\leq \varepsilon 
\end{equation*}
for all $m,n\geq \pi(\varepsilon/2)$, and thus $\seq{s_n}$ has a computable rate of Cauchy convergence, which is false.
\end{proof}
\begin{proposition}
\label{prop-tauberspecker}
There exists $\seq{a_n}$ satisfying the premise of Theorem \ref{thm-taubcauchy}, whereby though $\lim_{m,n\to\infty}|F(v_m)-s_n|=0$ this has no computable rate of convergence.
\end{proposition}

\begin{proof}
Take any Specker sequence $\seq{q_n}$ and define $a_0:=q_0$, $a_1:=q_1-q_0$ and for $n\geq 2$
\begin{equation*}
a_{n}:=\frac{q_{m+1}-q_m}{2^{m-1}} \; \; \; \mbox{for $m:=\ceil{\log_2(n)}$}. 
\end{equation*}
We first observe that since $2^{m-1}\geq n/2$ we have
\begin{equation*}
n|a_n|=\frac{n}{2^{m-1}}(q_{m+1}-q_m)\leq 2(q_{m+1}-q_m)\to 0
\end{equation*}
as $n\to\infty$, and thus $a_n=o(1/n)$. An easy induction establishes that $s_{2^n}=q_{n+1}$ for all $n\geq 1$, where for the induction step we have
\begin{equation*}
s_{2^n}=s_{2^{n-1}}+\sum_{i=2^{n-1}+1}^{2^n}a_i=q_n+\sum_{i=2^{n-1}+1}^{2^n}\left(\frac{q_{n+1}-q_n}{2^{n-1}}\right)=q_n+(q_{n+1}-q_n)=q_{n+1}.
\end{equation*}
This implies that $\lim_{n\to\infty} s_n=\lim_{n\to\infty} q_n$, and so in particular by Abel's theorem $\lim_{x\to 1^{-}} F(x)$ exists and so $\seq{F(v_n)}$ is Cauchy. Therefore $\seq{a_n}$ satisfies the premise of Theorem \ref{thm-taubcauchy}. But now suppose by contradiction we have a computable rate of convergence $\pi:\QQ_+\to\NN$ for $\lim_{m,n\to\infty}|F(v_m)-s_n|=0$. Then as in the proof of Proposition \ref{prop-abelspecker}, for any $\varepsilon\in\QQ_+$ we have $|s_m-s_n|\leq \varepsilon$ for all $m,n\geq \pi(\varepsilon/2)$, and therefore $|q_m-q_n|\leq \varepsilon$ for all $m,n\geq \phi(\varepsilon)$, where $\phi(\varepsilon):=\ceil{\log_2(\pi(\varepsilon/2))}+1$. But since $\phi$ is computable, this contradicts the assumption that $\seq{q_n}$ is a Specker sequence.
\end{proof}

\section{Finitizations of Theorems \ref{thm-abelcauchy} and \ref{thm-taubcauchy}}
\label{sec-finitize}

We now give finitary, or quantitative formulations to our Cauchy variants of Abel and Tauber's theorems, given in each case as a route from a metastable version of the premise to that of the conclusion. These results are finitary in the sense that aside from a global bound on our input data, we only consider finite initial segments of this data, and quantitative in the sense that they provide an explicit method for constructing rates of metastability for the conclusion in terms of rates of metastability from the premises. Moreover, the proofs of both results are entirely finitistic in nature, appealing to nothing more than simple arithmetic operations. In what follows, in addition to our notational conventions established in Section \ref{sec-cauchy}, for $l\in\NN$ we define $F_l:[0,1)\to\RR$ by
\begin{equation*}
F_l(x):=\sum_{i=0}^l a_ix^i.
\end{equation*}
In the remainder of this paper, we denote by $\omega:\QQ_+\times\NN_{>0}\to\NN$ some canonical computable function satisfying, for all $\varepsilon\in\QQ_+$ and $p\geq 1$: 
\begin{itemize}

\item $\omega(\varepsilon,p)\geq p$,

\item if $x\in [0,1-\tfrac{1}{p}]$ then $x^l\leq\varepsilon$ whenever $l\geq \omega(\varepsilon,p)$.

\end{itemize}
For instance, using the standard inequality 
\begin{equation*}
(1+y)^r\leq e^{yr} \ \ \ \ \ y\in\RR,r>0
\end{equation*}
we could set $\omega(\varepsilon,p):=p\cdot\ceil{\log(1/\varepsilon)}$, but for notational simplicity we work directly with $\omega$ rather than any specific function. Our first result recalls that a power series has a computable rate of convergence within any compact interval $[0,1-\frac{1}{p}]\subset [0,1)$ given a bound on its coefficients, and thus the function $F$ can be approximated by $F_l$ for $l$ computable in the desired degree of accuracy. 
\begin{lemma}
\label{lem-powerconv}
Let $\seq{|a_n|}$ be bounded above by some $L\in\NN$. Then for any $\varepsilon\in\QQ_+$ and $p\geq 1$ we have 
\begin{equation*}
|F(x)-F_l(x)|\leq \varepsilon
\end{equation*}
whenever $x\in [0,1-\tfrac{1}{p}]$ and $l\geq \omega(\frac{\varepsilon}{Lp},p)$.
\end{lemma}

\begin{proof}
To see this, we simply observe that
\begin{equation*}
\norm{F(x)-F_l(x)}=\norm{\sum_{i=l+1}^\infty a_ix^i}\leq \sum_{i=l+1}^\infty |a_i| x^i\leq L\left(\frac{x^{l+1}}{1-x}\right)\leq Lpx^{l+1}\leq Lpx^l\leq \varepsilon
\end{equation*}
where in the last step we use the defining property of $\omega$.\end{proof}
We now present our finitary theorems, where we recall the notation 
\begin{equation*}
[n;k]:=\{n,n+1,\ldots,k-1,k\}
\end{equation*}
for $n\leq k$, and just $[n;k]:=\emptyset$ for $k<n$.
\begin{theorem}
[Finite Abel's theorem]
\label{thm-abelfin}
Let $\seq{a_n}$ and $\seq{x_k}$ be arbitrary sequences of reals, and $L\in\NN$ a bound for $\seq{|s_n|}$. Fix some $\varepsilon\in\QQ_+$ and $g:\NN\to\NN$. Suppose that $N_1,N_2\in\NN$ and $p\geq 1$ are such that
\begin{equation*}
|s_i-s_n|\leq \frac{\varepsilon}{4} \ \ \ \mbox{and} \ \ \ \frac{1}{p}\leq 1-x_m \leq \frac{\varepsilon}{8LN_1}
\end{equation*}
for all $i,n\in [N_1;\max\{N+g(N),l\}]$ and all $m\in [N_2;N+g(N)]$ where 
\begin{itemize}

\item $\displaystyle N:=\max\{N_1,N_2\}$,

\item $\displaystyle l:=\omega\left(\frac{\varepsilon}{8Lp},p\right)$.

\end{itemize}
%
%
Then we have $|F(x_m)-s_n|\leq \varepsilon$ for all $m,n\in [N;N+g(N)]$.
\end{theorem}

\begin{proof}
Fix some $m,n\in [N;N+g(N)]$. We first note that since $m\geq N$ then $m\geq N_2$ and thus $m\in [N_2;N+g(N)]$ from which it follows that $2LN_1(1-x_m)\leq \tfrac{\varepsilon}{4}$. Using this together with the fact that $n\in [N_1;\max\{N+g(N),l\}]$ and thus $|s_i-s_n|\leq\tfrac{\varepsilon}{4}$ for any $N_1\leq i\leq l\leq \max\{N+g(N),l\}$ we have
\begin{equation}
\label{eqn-abelfin1}
\begin{aligned}
\norm{(1-x_m)\sum_{i=0}^{l-1}(s_i-s_n)x_m^i}&\leq (1-x_m)\sum_{i=0}^{l-1}|s_i-s_n| x_m^i\\
&\leq (1-x_m)\sum_{i=0}^{N_1-1}|s_i-s_n|x_m^i+(1-x_m)\sum_{i=N_1}^{l-1}|s_i-s_n|x_m^i\\
&\leq (1-x_m)\sum_{i=0}^{N_1-1}(|s_i|+|s_n|)+(1-x_m)\cdot \tfrac{\varepsilon}{4}\sum_{i=N_1}^{l-1}x_m^i\\
&\leq 2LN_1(1-x_m)+\tfrac{\varepsilon}{4} \cdot(x_m^{N_1}-x_m^l)\leq \tfrac{\varepsilon}{4}+\tfrac{\varepsilon}{4}=\tfrac{\varepsilon}{2}.
\end{aligned}
\end{equation}
Next, using the fact that for any $x$ and $l$ we have $$F_l(x)=s_lx^l+(1-x)\sum_{i=0}^{l-1} s_ix^i,$$ together with (\ref{eqn-abelfin1}) we see that
\begin{equation}
\label{eqn-abelfin2}
\begin{aligned}
|F_l(x_m)-s_n|&=\norm{s_lx_{m}^l+(1-x_m)\sum_{i=0}^{l-1}s_ix_m^i-s_n}\\
&\leq |s_lx_m^l|+\norm{(1-x_m)\sum_{i=0}^{l-1}(s_i-s_n)x_m^i}+\norm{(1-x_m)\sum_{i=0}^{l-1}s_nx_m^i-s_n}\\
&\leq |s_l|x_m^l+\tfrac{\varepsilon}{2}+|s_n|x_m^l\leq \tfrac{3\varepsilon}{4}
\end{aligned}
\end{equation}
where for the last step we use that $(|s_l|+|s_n|)x_m^l\leq 2Lx_m^l\leq \frac{\varepsilon}{4p}\leq \frac{\varepsilon}{4}$ which holds by definition of $l$ together with the fact that $x_m\in [0,1-\tfrac{1}{p}]$. Finally, observing that $|a_j|=|s_j-s_{j-1}|\leq |s_j|+|s_{j-1}|\leq 2L$ for any $j\in\NN$, by Lemma \ref{lem-powerconv} and (\ref{eqn-abelfin2}) we have
\begin{equation*}
|F(x_m)-s_n|\leq |F(x_m)-F_l(x_m)|+|F_l(x_m)-s_n|\leq\tfrac{\varepsilon}{4}+\tfrac{3\varepsilon}{4}\leq \varepsilon,
\end{equation*}
which completes the proof.
\end{proof}

\begin{theorem}
[Finite Tauber's theorem]
\label{thm-taubfin}
Let $\seq{a_n}$ be an arbitrary sequence of reals, and $L$ a bound for $\seq{|a_n|}$. Define $v_n:=1-\tfrac{1}{n}$, and fix some $\varepsilon\in\QQ_+$ and $g:\NN\to\NN$. Suppose that $N_1,N_2\in\NN$ are such that
\begin{equation*}
i|a_i|\leq\frac{\varepsilon}{8} \ \ \ \mbox{and} \ \ \ |F(v_m)-F(v_n)|\leq\frac{\varepsilon}{4}
\end{equation*}
for all $i\in [N_1;l]$ and all $m,n\in [N_2;N+g(N)]$ where 
\begin{itemize}

\item $\displaystyle N:=\max\left\{\frac{2LN_1^2}{\varepsilon},N_2\right\}$,

\item $\displaystyle l:=\omega\left(\frac{\varepsilon}{4Lp},p\right)$ for $p:=N+g(N)$.

\end{itemize}
%
%
Then we have $|F(v_m)-s_n|\leq \varepsilon$ for all $m,n\in [N;N+g(N)]$.
\end{theorem}

\begin{proof}
Fix some $m,n\in [N;N+g(N)]$. We first note that since $\frac{2LN_1^2}{\varepsilon}\leq n$ then we have $\frac{LN_1^2}{2n}\leq \frac{\varepsilon}{4}$, and since $N_1\leq \tfrac{2LN_1^2}{\varepsilon}\leq n\leq l$ then for any $N_1\leq i\leq n$ we have $i|a_i|\leq\frac{\varepsilon}{8}$. Therefore
\begin{equation}
\label{eqn-taubfin0}
\begin{aligned}
|F_n(v_n)-s_n|&\leq \sum_{i=0}^n |a_i|(1-v_n^i)\leq \sum_{i=0}^n i|a_i|(1-v_n)=\frac{1}{n}\sum_{i=0}^n i|a_i|\\
&=\frac{1}{n}\sum_{i=0}^{N_1-1}i|a_i|+\frac{1}{n}\sum_{i=N_1}^n i|a_i|\leq \frac{L}{n}\cdot\frac{1}{2}(N_1-1)N_1+\frac{\varepsilon}{8n}(n-N_1)\\
&\leq \frac{LN_1^2}{2n}+\frac{\varepsilon}{8}\leq \frac{3\varepsilon}{8}
\end{aligned}
\end{equation}
where for the second step we use that $(1-x^i)\leq i(1-x)$. Similarly, we have
\begin{equation}
\label{eqn-taubfin1}
\begin{aligned}
|F_l(v_n)-F_n(v_n)|&=\norm{\sum_{i=n+1}^l a_iv_n^i}\leq \frac{\varepsilon}{8}\sum_{i=n+1}^l \frac{v_n^i}{i}\\
&\leq \frac{\varepsilon}{8(n+1)}\sum_{i=n+1}^l v_n^i \leq \frac{\varepsilon(v_n^{n+1}-v_n^{l+1})}{8(n+1)(1-v_n)}\leq \frac{\varepsilon n}{8(n+1)}\leq\frac{\varepsilon}{8}.
\end{aligned}
\end{equation}
Now, since for any $n\in [N;N+g(N)]$ we have $v_n\in [0,1-\tfrac{1}{p}]$ for $p:=N+g(N)$ by definition, it follows by Lemma \ref{lem-powerconv} that $|F(v_n)-F_l(v_n)|\leq \tfrac{\varepsilon}{4}$, and thus 
\begin{equation*}
\begin{aligned}
|F(v_m)-s_n|&\leq |F(v_m)-F(v_n)|+|F(v_n)-F_l(v_n)|+|F_l(v_n)-s_n|\\
&\leq \frac{\varepsilon}{4}+\frac{\varepsilon}{4}+|F_l(v_n)-F_n(v_n)|+|F_n(v_n)-s_n|\\
&\leq \frac{\varepsilon}{2}+\frac{\varepsilon}{8}+\frac{3\varepsilon}{8}\leq \varepsilon
\end{aligned}
\end{equation*}
where the last line follows from (\ref{eqn-taubfin0}) and (\ref{eqn-taubfin1}). This completes the proof.
\end{proof}

\section{Reobtaining the infinitary variants}
\label{sec-reobtain}

We conclude by showing how our the theorems of the previous section, though finitary in nature, are strong enough to allow us to reobtain the usual formulations of Abel and Tauber's theorems using purely logical reasoning. We first need a simple but crucial lemma which we use throughout this section.
\begin{lemma}
\label{lem-pi3}
Let $P(\varepsilon,X)$ be some predicate on $\varepsilon\in \QQ_+$ and finite subsets $X\subset \NN$. Then the following two statements are equivalent:
\begin{enumerate}[(a)]

\item $\forall \varepsilon\in \QQ_+\; \exists n\in \NN\; \forall k\; P(\varepsilon,[n;k])$,

\item $\forall \varepsilon\in \QQ_+\; \forall g:\NN\to \NN\; \exists n\in \NN\; P(\varepsilon,[n;g(n)])$. 

\end{enumerate}
\end{lemma}
\begin{proof}
For $(a)\Rightarrow (b)$, if for some $\varepsilon\in \QQ_+$ there is some $n\in\NN$ satisfying $P(\varepsilon,[n;k])$ then in particular for any $g:\NN\to\NN$ we have $P(\varepsilon,[n;g(n)])$. To establish $(b)\Rightarrow (a)$ we suppose for contradiction that $(a)$ is false, and thus for some $\varepsilon\in \QQ_+$ it is the case that for $\forall n\in\NN\; \exists k\; \neg P(\varepsilon,[n;k])$. Therefore by the axiom of choice there is some $g:\NN\to\NN$ satisfying $\forall n\in\NN\; \neg P(\varepsilon,[n;g(n)])$, contradicting $(b)$.
\end{proof}
By setting $P(\varepsilon,X):\Leftrightarrow \forall m,n\in X(|s_m-s_n|\leq \varepsilon)$ the equivalence of Cauchy convergence and metastability in the sense of (\ref{eqn-metastable}) is a direct corollary of the above lemma -- note that the slightly different statement $\forall \varepsilon,g\exists n P(\varepsilon,[n;n+g(n)])$ is just another way of expressing (b). By extending Lemma \ref{lem-pi3} to the various other Cauchy properties involved in our finitary theorems, we are able to prove the original, infinitary variants of those theorems.
\begin{proof}[Proof of Theorem \ref{thm-abelcauchy} from Theorem \ref{thm-abelfin}]
Suppose that $\seq{a_n}$ and $\seq{x_m}$ are such that (i) $\seq{s_n}$ is Cauchy, (ii) $\lim_{m\to\infty} =1$. Note that since $\seq{s_n}$ convergences then $\seq{|s_n|}$ must be bounded above by some $L$. Now fix some arbitrary $\varepsilon\in\QQ_+$ and $g:\NN\to\NN$. From $\lim_{m\to\infty} x_m=1$ and Lemma \ref{lem-pi3} we can infer that for any $\delta>0$ and $h:\NN\to\NN$ there exists some $n\in\NN$ such that
\begin{equation*}
\forall m\in [n;h(n)](1-\delta\leq x_m).
\end{equation*}
Using a weak form of the axiom of choice, let $\Phi(\delta,h)$ be the functional which returns such an $n$ for any given $\delta$ and $h$, and define $f:\NN\to\NN$ by
\begin{equation*}
f(a):=\max\left\{M_a+g(M_a),\omega\left(\frac{\varepsilon}{8Lp_{a}},p_{a}\right)\right\}
\end{equation*}
where we define
\begin{equation*}
\begin{aligned}
M_a&:=\max\left\{a,\Phi(\frac{\varepsilon}{8La},h_a)\right\}\\[2mm]
p_a&:=\ceil{\max\{1/(1-x_m)\; : \; m\leq M_a+g(M_a)\}}\\[2mm]
h_a(b)&:=\max\{a,b\}+g(\max\{a,b\}).
\end{aligned}
\end{equation*}
Now, from Cauchyness and hence metastability of $\seq{s_n}$ we can infer that there exists some $N_1\in\NN$ such that
\begin{equation*}
\forall i,n\in [N_1;f(N_1)](|s_i-s_n|\leq\tfrac{\varepsilon}{4}).
\end{equation*}
Define $N_2:=\Phi(\frac{\varepsilon}{8LN_1},h_{N_1})$, so that
\begin{equation*}
\forall m\in [N_2;h_{N_1}(N_2)]\left(1-\frac{\varepsilon}{8LN_1}\leq x_m\right).
\end{equation*}
Then setting $N:=\max\{N_1,N_2\}$ we see that $M_{N_1}=N$ and therefore $$f(N_1)=\max\{N+g(N),l\}$$ for $l=\omega(\frac{\varepsilon}{8Lp},p)$ with $p=\ceil{\max\{1/(1-x_m)\; : \; m\leq N+g(N)\}}$, and in addition $$h_{N_1}(N_2)=N+g(N).$$ Observing finally that for $m\leq N+g(N)$ we must have $1/(1-x_m)\leq p$ and thus $x_m\leq 1-\frac{1}{p}$, we see that $N_1,N_2$ and $p$ satisfy the premise of Theorem \ref{thm-abelfin}, and therefore $|F(x_m)-s_n|\leq\varepsilon$ for all $m,n\in [N,N+g(N)]$. But since $\varepsilon$ and $g$ were arbitrary, by Lemma \ref{lem-pi3} this means that $\lim_{m,n\to\infty}|F(x_m)-s_n|=0$.
\end{proof}

\begin{proof}[Proof of Theorem \ref{thm-taubcauchy} from Theorem \ref{thm-taub}]
Suppose that $\{a_n\}$ is a sequence of reals such that (i) $a_n=o(1/n)$ and (ii) $\seq{F(v_m)}$ is Cauchy. From (i) we must have that $\seq{|a_n|}$ is bounded above by some $L$. Now fix some arbitrary $\varepsilon\in\QQ_+$ and $g:\NN\to\NN$. From $\lim_{n\to\infty} n|a_n|=0$ and Lemma \ref{lem-pi3} we can infer that for any $h:\NN\to\NN$ there exists some $n\in\NN$ such that
\begin{equation*}
\forall i\in [n;h(n)](i|a_i|\leq \tfrac{\varepsilon}{8}).
\end{equation*}
Let $\Psi(h)$ be the functional which returns such an $n$ for any given $h$, and define $f:\NN\to\NN$ by
\begin{equation*}
f(a):=\omega\left(\frac{\varepsilon}{4Lp_a},p_a\right)
\end{equation*}
where we define
\begin{equation*}
\begin{aligned}
p_a&:=M_a+g(M_a)\\
M_a&:=\max\left\{\frac{2La^2}{\varepsilon},\Psi(h_a)\right\}\\
h_a(b)&:=\max\left\{\frac{2La^2}{\varepsilon},b\right\}+g\left(\max\left\{\frac{2La^2}{\varepsilon},b\right\}\right).
\end{aligned}
\end{equation*}
From Cauchyness and hence metastability of $\seq{F(v_n)}$ we infer that there exists some $N_1\in\NN$ such that
\begin{equation*}
\forall m,n\in [N_1;f(N_1)](|F(v_m)-F(v_n)|\leq\tfrac{\varepsilon}{4}).
\end{equation*} 
Define $N_2:=\Psi(h_{N_1})$, so that
\begin{equation*}
\forall i\in [N_2;h_{N_1}(N_2)](i|a_i|\leq \tfrac{\varepsilon}{8}).
\end{equation*}
Then setting $N:=M_{N_1}=\max\{\frac{2LN_1^2}{\varepsilon},N_2\}$ we see that $$f(N_1)=\omega\left(\frac{\varepsilon}{4Lp},p\right)$$ for $p=N+g(N)$ and $$h_{N_1}(N_2)=N+g(N),$$ and therefore $N_1$ and $N_2$ satisfy the premise of Theorem \ref{thm-taubfin}. Therefore $|F(v_m)-s_n|\leq\varepsilon$ for all $m,n\in [N;N+g(N)]$, and since $\varepsilon$ and $g$ were arbitrary, this means by Lemma \ref{lem-pi3} that $\lim_{m,n\to\infty}|F(v_m)-s_n|=0$.
\end{proof}
%
\section{General proof theoretic remarks}
\label{sec-prooftheory}

The main quantitative results in this paper were obtained by carrying out an analysis of the original proofs of both Abel's and Tauber's theorems using the classical G\"{o}del Dialectica interpretation (i.e. a combination of the usual Dialectica interpretation with a negative translation). In both cases, the resulting realizing terms were simple enough that the core \emph{combinatorial part} of the analysis could be presented in a traditional mathematical style, in particular without reference to higher-order functionals. This gave rise to Section \ref{sec-finitize} and Theorems \ref{thm-abelfin} and \ref{thm-taubfin}.

The full the analysis, in which higher-order rates of metastability for the conclusions of each theorem are constructed in terms of corresponding rates for the premises, follows by appealing to the results of Section \ref{sec-reobtain}. In particular, the routes from Theorem \ref{thm-abelfin} to \ref{thm-abelcauchy} resp. Theorem \ref{thm-taubfin} to \ref{thm-taubcauchy} use, in both cases, a simple form of \emph{finite bar recursion} of length two. Intuitively, this corresponds to the fact that in the standard proofs of Abel's resp. Tauber's theorem, the two main premises are combined via a form of \emph{bounded collection}, which can in turn interpreted by finite bar recursion (see also \cite{EOP(2011.0)}). Though in this paper we do not work in any formal systems, the proofs of Section \ref{sec-reobtain} essentially use just classical predicate logic together with the quantifier-free axiom of choice, and in this sense, our finitary theorems imply the infinitary versions over a very weak base theory, and using just elementary logical reasoning.

In this paper our priority was to establish finitary formulations of well-known convergence results rather than present a detailed application of the Dialectica interpretation, which is why we have kept much of the formal proof theory suppressed. Nevertheless, we conclude by sketching a simple example of how our quantitative formulation of Abel's theorem can be used to obtain concrete rates of metastability in the style of traditional proof mining. 

Let us take the following simple consequence of Abel's theorem, which follows directly from Theorem \ref{thm-abelcauchy}, using the fact that whenever $\seq{a_n}$ is a sequence of \emph{positive} reals then $\seq{s_n}$ is monotonically increasing, and thus is Cauchy whenever it is bounded above:
\begin{proposition}
Let $\seq{a_n}$ be a sequence of positive reals such that the $\seq{s_n}$ are bounded above. Then $\lim_{m,n\to\infty}|F(v_m)-s_n|=0$.
\end{proposition}
An analysis of this result using Theorem \ref{thm-abelfin} together with ideas from Section \ref{sec-reobtain} yields the following.
\begin{proposition}
Let $\seq{a_n}$ be a sequence of positive reals and $L$ a bound for the $\seq{s_n}$. Then for any $\varepsilon\in \QQ_+$ and $g:\NN\to\NN$ we have
\begin{equation*}
\exists N\leq \Gamma_L(\varepsilon,g)\; \forall m,n\in [N,N+g(N)]\; (|F(v_m)-s_n)|\leq\varepsilon)
\end{equation*}
for
\begin{itemize}

\item $\displaystyle \Gamma_L(\varepsilon,g):=\Bigl\lceil\frac{8Lf^{(\ceil{4L/\varepsilon})}(0)}{\varepsilon}\Bigr\rceil$

\item $f(a):=p_a\cdot \Bigl\lceil\log\left(\frac{8Lp_a}{\varepsilon}\right)\Bigr\rceil$

\item $p_a:=\tilde g\left(\Bigl\lceil\frac{8La}{\varepsilon}\Bigr\rceil\right)$

\end{itemize}
where $f^{(k)}(x)$ denotes the $k$-times iteration of $f$ applied to $x$, and $\tilde g(x)$ is defined by $\tilde g(x):=x+g(x)$.
\end{proposition}

\begin{proof}
Following closely and using notation from the proof of Theorem \ref{thm-abelcauchy} from Theorem \ref{thm-abelfin} in Section \ref{sec-reobtain}, we first note that setting $x_m:=v_m=1-\tfrac{1}{m}$, for any $\delta>0$ and $h:\NN\to\NN$ we trivially have
\begin{equation*}
\forall m\in [n;h(n)](1-\delta\leq v_m)
\end{equation*}
for $n:=\ceil{1/\delta}$, and thus we can define $\Phi(\delta,h):=\ceil{1/\delta}$. Therefore in this case, $M_a=\max\{a,\ceil{8La/\varepsilon}\}=\ceil{8La/\varepsilon}$ and $p_a=M_a+g(M_a)=\tilde g(\ceil{8La/\varepsilon})$, and thus using our explicit definition of $\omega(\varepsilon,p)=p\cdot\ceil{\log(1/\varepsilon)}$ from Section \ref{sec-finitize}, we see that
\begin{equation*}
f(a)=\max\{p_a,\omega(\varepsilon/8Lp_a,p_a)\}=p_a\cdot \ceil{\log(8Lp_a/\varepsilon)}.
\end{equation*}
Now, it is a well-known fact from proof mining (cf. \cite[Proposition 2.26]{Kohlenbach(2008.0)}) that for any monotone increasing $\seq{s_n}$ bounded above by some $L$, a bound on the corresponding rate of metastability i.e.
\begin{equation*}
\forall \varepsilon'\in \QQ_+\; \forall f':\NN\to\NN\; \exists N'\leq \Psi(\varepsilon',f') \forall m\in [N';f'(N')](|s_m-s_n|\leq\varepsilon')
\end{equation*}
is given by $\Psi(\varepsilon',f'):={f'}^{\ceil{(L/\varepsilon)}}(0)$. Thus in this case, setting $\varepsilon':=\frac{\varepsilon}{4}$ and $f'=f$ we would have $N_1\leq f^{\ceil{(4L/\varepsilon)}}(0)$ and therefore
\begin{equation*}
N_2:=\Phi\left(\frac{\varepsilon}{8LN_1},h_{N_1}\right)=\Bigl\lceil \frac{8LN_1}{\varepsilon} \Bigr\rceil\leq \Bigl\lceil \frac{8Lf^{\ceil{(4L/\varepsilon)}}(0)}{\varepsilon} \Bigr\rceil.
\end{equation*}
Therefore by Theorem \ref{thm-abelfin}, $N:=\max\{N_1,N_2\}=N_2$ satisfies
\begin{equation*}
\forall m,n\in [N;N+g(N)]\; (|F(v_m)-s_n)|\leq\varepsilon).
\end{equation*}
Backtracking through the above definitions yields the given bound on $N$.
\end{proof}


\bibliographystyle{plain}
\bibliography{/home/thomas/Dropbox/tp}

\end{document}